\documentclass[12pt,reqno,oneside]{amsart}

\usepackage{amsmath,amsthm,amsfonts,amssymb}
\usepackage{indentfirst}
\usepackage{url}
\usepackage{graphicx}
\usepackage[usenames]{color}

\theoremstyle{plain}
\newtheorem{theorem}{Theorem}[section]

\newtheorem{lem}[theorem]{Lemma}

\theoremstyle{definition}
\newtheorem{defn}[theorem]{Definition}

\newtheorem{obs}[theorem]{Remark}

\numberwithin{equation}{section}
\numberwithin{figure}{section}

\newcommand{\bbP}{\mathbb{P}}
\newcommand{\bbE}{\mathbb{E}}
\newcommand{\bbT}{\mathbb{T}}

\begin{document}

\baselineskip=18pt

\title[Evolution with mass extinction on $\bbT_d^+$]{Evolution with mass extinction on $\bbT_d^+$}

\author[{\color{blue}Carolina Grejo}]{Carolina Grejo}
\address[C. Grejo]{Institute of Mathematics and Computer Science, Universidade de S\~ao Paulo,  Brazil.}
\email{carolina@ime.usp.br}

\author[F\'abio Lopes]{F\'abio Lopes}
\address[F. Lopes]{Departamento de Matem\'atica, Universidad Tecnol\'ogica Metropolitana, Chile}
\email{f.marcellus@utem.cl}

\author[F\'abio Machado]{F\'abio Machado}
\address[F. Machado]{Institute of Mathematics and Statistics, Universidade de S\~ao Paulo, Brazil.}
\email{fmachado@ime.usp.br}

\author[Alejandro Roldan]{Alejandro Rold\'an-Correa}
\address[A. Roldan]{Instituto de Matematicas, Universidad de Antioquia, Colombia}
\email{alejandro.roldan@udea.edu.co}

\thanks{Research supported by CAPES(001), CONICYT/FONDECYT Postdoctorado (3160163), CNPq(303699/2018-3), FAPESP(17/10555-0) and Universidad de Antioquia.}

\date{\today}

\begin{abstract} 
We propose a stochastic model for evolution through mutation and natural selection of a population that evolves on a $\bbT_d^+$ tree. We think of this model as a way of describing the evolution fitness landscape of a population. We obtain sharp and distinct conditions on the set of parameters for extinction and survival.
\end{abstract}

\maketitle

\section{Introduction}
\label{S: Introduction}

We propose a stochastic spatial model for evolution through mutation and natural selection.  The proposed model is 
an interacting particle system evolving on $\bbT_d^+$
(an infinite rooted tree whose vertices have degree $d+1$, except the root that has degree $d$) whose state space is $\{0,1\}^{\bbT_d^+}$. That is, the status of a vertex can be either ``0'' (empty) or ``1'' (occupied  by a particle).  The set of vertices that are at a distance  $i$ from the root is namely the  level $i$ of the tree.
 At time $t=0$ there is only one particle in the system, located at the root of 
$\bbT_d^+$. This particle will have a deterministic lifetime of length 1. During its lifetime, following a Poisson process of rate $\lambda$, it generates offspring (new particles) which will be placed randomly, one by one, on its $d$ nearest neighbours up to the time it dies or up to the event that all its nearest neighbours are occupied. 
In general, each particle which are placed at level $i$ of ${\bbT_d^+}$ may generate offspring according to an independent Poisson process with rate $\lambda$. These offspring will be placed randomly, one by one, at one of the $d$ nearest neighbours ($(i+1)-th$ level) up to the time their progenitor dies or up to the event that all these $d$ nearest neighbours are occupied. With respect to the death time, all individuals placed at level $i$ follow the same clock. However, {this clock (of length 1) does not start to tick until their progenitors (placed at level $(i-1)$) dies out.  We call this the residual lifetime. Note that while individuals placed at a given level have different lifetimes they die (and are removed) at the same time. Besides that their residual lifetime starts and ends at the same moment. We call this a mass extinction event. 

We think of our model as describing the fitness landscape of a population. We consider the offspring as beneficial mutations that occurs in the population and the level where they are placed, their fitness. Here $(i)$ a mutation is not at risk while its progenitor 
is alive and $(ii)$ a death event at a given time, kills all the least fit mutations (particles placed at the lower level) 
present in the process. It combines features that have been recently explored in population evolutionary dynamics as it assumes $(i)$ the occurrence of mass extinction events that may lead all individuals of certain type in the population to extinction as in Schinazi and Schweinsberg~\cite{schinazi} and $(ii)$ that mass extinction events kill the least fit type of individuals in the population as in Guiol \textit{et al.} \cite{GMS11}.  Figure~\ref{fig:1} shows an illustration of the model on $\bbT_2^+$.\\

\begin{figure}[h!]
	\centering
	\includegraphics[width=\textwidth]{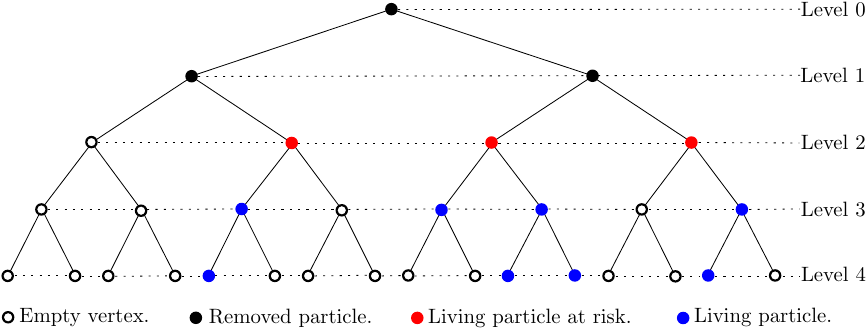}
	\caption{Illustration of the model on $\bbT_2^+$ with vertex depth up to 4.}
	\label{fig:1}
\end{figure}

Our model is a variation of the birth-and-assassination (BA) process introduced by Aldous and Krebs  \cite{aldous}.The BA process  is a variant the classical branching process. 	It was stated to investigate the scaling limit of a certain queuing system with blocking which appeared in database processing, see Tsitsiklis \textit{et al}~\cite {Tsitsiklis}. 
In the BA process the clocks that define the residual lifetime of each offspring of a particle start to tick as soon as the lifetime of that particle comes to an end, independently from each	other. Another difference in BA process is that each particle which are placed at level $i$ may generate offspring according to an independent Poisson process with rate $\lambda$ and no restriction. Our model has a stronger dependence among particles lifetime,
as all particles at the same level share the same clock as for the death time.

Our model is also related to spatial catastrophe models where large sets of individuals are simultaneously removed, see for example Lanchier \cite{L2011} and Machado  \textit{et al.} \cite{MRS17}. 
We can also imagine the particles in this process as pathogens in such way 
that if the process  dies out, the disease is defeated.  

The aim of this paper is to establish the relationship between the birth and death rates in the survival or extinction of infectious process. Hence, we  found conditions under which particles are eventually removed from $\bbT_d^+$.

\begin{defn} 
If all particles are eventually removed from $\bbT_d^+$ with probability 1, we say that the process 
{\it dies out}. Otherwise, we say that the process {\it survives}.
\end{defn}

Let us introduce the notation $ \eta_t \in \{0,1\}^{\bbT_d^+}$ for the status of the vertices in terms of occupation at time $t$ and $|\eta_t|$ 
for the amount of particles present at time $t$. By coupling arguments, one can see that the pro\-bability of survival is a non-decreasing function of $\lambda$. This is so because more births (greater $\lambda$) can only help survival. So we define

\[ \lambda_c(d) := \inf \{ \lambda: \bbP(|\eta_t| \ge 1 \hbox{ for all } t \ge 0) > 0 \}. \]
 
\section{Results}

We prove phase transition (meaning that $0 < \lambda_c(d) < \infty$)  for this process.

\begin{theorem}
\label{T:arvore} For $d$ fixed 
$$\lambda_c(d)=\inf\left\{\lambda: \inf_{u>0}\frac{\lambda e^u}{u}\left[1-\left(\frac{\lambda}{\lambda +u}\right)^d\right]>1\right\}.$$
If $\lambda<\lambda_c(d)$ then the process $\eta_t$ dies out. If $\lambda>\lambda_c(d)$  then the process $\eta_t$ survives.
\end{theorem}

\begin{obs} Let us define  
\[ f_d(\lambda):=\inf_{u>0}\frac{\lambda e^u}{u}\left[1-\left(\frac{\lambda}{\lambda +u}\right)^d\right]. \]

The function $f_d(\lambda)$ is continuous and strictly monotone in $(0,1)$, $f_d(0)=0$
and $f_d(1)>1$. Thus, by Theorem~\ref{T:arvore}, we have that $\lambda_c(d)$ is the unique solution for  $f_d(\lambda)=1$.

Table~\ref{tab:adicaoZ4} shows some numerical approximations for $\lambda_c(d)$. 

\begin{table}[htb]
	\centering
	\large
\begin{tabular}{|c|c|c|c|c|c|c|c|c|}\hline
$d$ & 2 &3  & 4     &  5  & 6 & 7           \\ \hline
$\lambda_c(d)$ & 0.40564 & 0.37596 &0.36990 &0.36839 & 0.36800 & 0.36792   \\ \hline
\end{tabular}\vspace{0.4cm}
	\caption{Numerical approximations for  $\lambda_c(d)$.}
\label{tab:adicaoZ4}
\end{table}

From the fact that $\lim_{d \to \infty} f_d(\lambda) = e\lambda$ and adap\-ting the
proof of Proposition 3.1 in Junior {\it et al.}~\cite{JMR16} one can see that
\begin{equation}
\label{lambdac}
 \lambda_c(d)\rightarrow 1/e.
\end{equation}
\end{obs}

\section{Proofs}

The proof below adapts a strategy presented in Aldous and Krebs \cite{aldous} for the survival and extinction of the birth-and-assassination (BA) process.  
We nevertheless include all the details here for completeness.
\begin{proof}[Proof Theorem \ref{T:arvore}]

First, we show that, the process goes extinct with probability 1, if $\inf_{u>0}\left\{\frac{\lambda e^u}{u}\left[1-\left(\frac{\lambda}{\lambda +u}\right)^d\right]\right\}<1$.\\
 Let $X_i$, $i=1,\ldots, d$, denote independent Gamma random variables with parameters $i$ and $\lambda$, respectively, and consider a random variable $W$ with probability distribution given by  
	$$\mathbb{P}(W\leq w)=\frac{1}{d}\left[\sum_{i=1}^d\mathbb{P}(X_i\leq w)\right],~w\geq 0.$$
	Note that, for $u>0$,
	\begin{eqnarray*}
	\mathbb{E}\left[e^{-uW}\right]&=&\frac{1}{d}\left[\frac{\lambda}{\lambda+u}+\left(\frac{\lambda}{\lambda +u}\right)^2+\cdots+\left(\frac{\lambda}{\lambda +u}\right)^d\right]\\
	&=&\frac{\lambda}{ud}\left[1-\left(\frac{\lambda}{\lambda +u}\right)^d\right].
	\end{eqnarray*}

	In our model, each time a particle occupying a vertex at the $k$-th level of $\bbT_d^+$ has an offspring, it chooses uniformly at random a position to place its offspring among its empty neighboring vertices at the $(k+1)$-th level of $\bbT_d^+$; once no empty vertex is left all incoming  particles are ignored. 
	In the absence of any information on these births and the ordering of the occupancy of the neighboring vertices at the $(k+1)$-th level, the time which is necessary for a fixed vertex at the $(k+1)$-th level to receive an offspring of a particle on its neighboring vertex at the $k$-th level has the same distribution of a random variable $W$. To see this, note that the fixed vertex at $(k + 1)$-th level   is occupied by the $i$-th offspring of its neighbor at $k$-th level with probability $1 / d$. The time for this to occur has Gamma distribution with parameters $(i, \lambda)$. Of course, such `virtual' births only become `real' ones, if the particle at the $k$-th level has these births before the killing event that removes all particles at the $k$-th level.\\
	Let $\{W_i\}$ be independent copies of the random variable $W$. Suppose that no information is known on the genealogy of a vertex at the $k$-th level in a given lineage of $\bbT_d^+$. Then, the probability that a particle is born at a given vertex at the $k$-th level of $\bbT_d^+$ is equal to

	$$\mathbb{P}\left(\sum_{i=1}^{j} W_i < \textcolor{blue}{j},~j=1,...,k\right),$$	

Note that, this probability is the same for every vertex at the $k$-th level and that, there are $d^k$ such vertices. Thus,\\

\noindent
$\mathbb{E}\left(\mbox{total number of particles born in the process}\right)=$
\begin{eqnarray*}\label{cheby}
&=&\sum_{v\in \bbT_d^+}\mathbb{E}[\mathbb{I} \{\mbox{a particle is born at}~v\}] \\
&=&\sum_{k\geq 1}d^k\mathbb{P}\left(\sum_{i=1}^j W_i < \textcolor{blue}{j},~j=1,...,k\right)  \\
&\leq &\sum_{k\geq 1}d^k\mathbb{P}\left(\sum_{i=1}^k W_i < \textcolor{blue}{k}\right) \\
&\leq &\sum_{k \geq 1}{d^k} \mathbb{E}\left(e^{u(\textcolor{blue}{k}-\sum_{i=1}^k W_i)}\right),  \\
&= &{\sum_{k\geq 1}\left[\frac{\lambda e^u}{u}\left[1-\left(\frac{\lambda}{\lambda +u}\right)^d\right]\right]^k,}
	\end{eqnarray*}

\noindent
where the second inequality is obtained by using Markov's inequality 
for $u>0$, and the last expression follows from independence and the properties of the moment generating functions of the random variables $\{W_i \}$.\\	Clearly, if {$\inf_{u>0}\left\{\frac{\lambda e^u}{u}\left[1-\left(\frac{\lambda}{\lambda +u}\right)^d\right]\right\}<1$}, then
\[\mathbb{E}\left(\mbox{total number of particles born in the process}\right)<+\infty.\] 
In this case, the process goes extinct since with probability 1 only a finite number of particles enters the system.\\
	
Next, we show that, when {$\inf_{u>0}\left\{\frac{\lambda e^u}{u}\left[1-\left(\frac{\lambda}{\lambda +u}\right)^d\right]\right\}>1$}, the process survives with positive probability. For this, we need the following  lemma of Aldous \& Krebs \cite{aldous}, which gives a large deviation estimate for the probability that a particle is born at the $k$-th level of $\bbT_d^+$.
\begin{lem}{\cite[Lemma 1]{aldous}}
\label{lemaldous}
Let $X_1, X_2, \ldots$ be i.i.d. random variables with $\mathbb{E}[X] <0$ and $\mathbb{P} [ X>0]>0$. Let $\mathbb{E}[e^{u X}]=\psi(u)$ be finite in some neighborhood of 0, and let $\rho=\inf_{u>0} \psi(u)$. Then,
$$ \lim_{n\rightarrow \infty}\frac{1}{n} \log \mathbb{P} \left[ \sum_{j=1}^k X_j>0, k=1,\ldots, n. \right]= \log \rho. $$ 
\end{lem}	
	Let $Z_i=\textcolor{blue}{1}-W_i$, $i=1,\ldots,n$. The probability that a particle is born at a fixed vertex at the $k$-th level can be rewritten as 

\begin{eqnarray*}\label{eq:Probverticefixo}	
		\mathbb{P}\left(\sum_{i=1}^j Z_i>0,~j=1,\ldots,k\right).
		\end{eqnarray*}
It is easy to check that, $\mathbb{E}[Z]<0$ if $\lambda <(d+1)/2$. So, by Lemma \ref{lemaldous},	

\[ \lim_{k\rightarrow\infty}\frac{1}{k}\log\mathbb{P}\left(\sum_{i=1}^j Z_i>0,~j=1,\ldots,k\right)= \hfill \] \[\hfill = \log\left[{\inf_{u>0}\left\{\frac{\lambda e^u}{ud}\left[1-\left(\frac{\lambda}{\lambda +u}\right)^d\right]\right\}} \right].\]

Moreover, by assumption, for some $\delta>0$,
$${\inf_{u>0}\left\{\frac{\lambda e^u}{u}\left[1-\left(\frac{\lambda}{\lambda +u}\right)^d\right]\right\}}=1+\delta.$$
Suppose for now that $\lambda <(d+1)/2$ and take $\epsilon=\delta/2$. Then, there exists $K\in\mathbb{N}$ such that for all $k\geq K$,\\	

$\mathbb{E} [\hbox{number of particles at $k$--th level of} ~\bbT_d^+]$
\begin{eqnarray*}\label{probpatogk}
&=&d^k\mathbb{P}\left(\sum_{i=1}^j Z_i>0,~j=1,\ldots, k\right) \nonumber \\
&>&d^k\left[ {\inf_{u>0}\left\{\frac{\lambda e^u}{du}\left[1-\left(\frac{\lambda}{\lambda +u}\right)^d\right]\right\}}-\frac{\epsilon}{d}\right]^k \nonumber\\
&=&\left[{\inf_{u>0}\left\{\frac{\lambda e^u}{u}\left[1-\left(\frac{\lambda}{\lambda +u}\right)^d\right]\right\}}-{\epsilon}\right]^k\\
&=&(1+\delta/2)^{k} >1.
\end{eqnarray*}

Note that each particle at $(n-1)k$--th level, may have at most $d^{nk}$  descendants at $nk$--th level. Moreover, the probability that a certain particle at level $(n-1)k$ has a descendant particle at $nk$--th level is greater than the probability that the particle originally located at the root has descendant at level $k$. This follows since by the time a particle at level $(n-1)k$ is born, it may not be at risk because some killing events corresponding to the previous levels
may still be pending. Therefore, the expected number of particles at level $nk$, descendants of a given particle at level $(n-1)k$, is also greater than 1.	
	
Now we define $\{Y_n\}_{n \ge 1}$ as an auxiliary process such that $Y_n$ is the number
of particles at level $nk$ born in the process $\eta_t$. Note that, from the definition of 
$\eta_t$, $Y_0=1$ and from the previous paragraph, $\{Y_n\}_{n \ge 1}$  dominates a 
Galton-Watson process with mean offspring $\mathbb{E}[Y_1]$. Since $\mathbb{E}[Y_1]>1$, the process $\{Y_n\}_{n \ge 1}$ survives with positive probability and, consequently, $\eta_t$ also does.

To conclude, a simple coupling argument can show that the survival probability is non-decreasing in $\lambda$. Hence, this result also holds for $\lambda\geq (d+1)/2$ provided {$\inf_{u>0}\left\{\frac{\lambda e^u}{u}\left[1-\left(\frac{\lambda}{\lambda +u}\right)^d\right]\right\}>1$}.
\end{proof}

\begin{obs}
Once the ancestors of a particle die, its residual lifetime is deterministic of length 1. Thus, the number of offspring generated by two different particles that are at the same tree level, during their residual lifetime, are independent. On the other hand, by the coupled death mechanism, this would not be true if residual lifetime were random. The independence property is important to guarantee that the auxiliary process $\{Y_n\}_{n\geq1}$ defined in the proof of Theorem~\ref{T:arvore} dominates a Galton-Watson process.
\end{obs}

\begin{obs}
Bordenave \cite{bordenave} shows that the BA process dies out at criticality for exponential residual lifetimes of rate 1. Actually he computes the expected total number of individuals born, $\bbE[N]$, showing at the critical value of the BA process, $\lambda_c = {1/4}$, it holds that $\bbE[N] = 2$. While this is a very interesting value, we could not extend the analytical approach of Bordenave to our model. We conjecture that extinction holds at the critical value,  $\lambda_c(d)$, for our model. 
\end{obs}


\end{document}